\documentclass[a4paper]{amsart}
\usepackage[T1]{fontenc}
\usepackage{amssymb}
\usepackage[all]{xy}
\usepackage{verbatim}

\newcommand{\abs}[1]{\lvert#1\rvert}

\newcommand{\bb}[1]{\mathbb{#1}}

\newcommand{\dotsk}[0]{,\dots,}

\newcommand{\f}[0]{k}                     

\newcommand{\bbN}[0]{\mathbb{N}}
\newcommand{\bbFp}[0]{\mathbb{F}_p}
\newcommand{\bbFq}[0]{\mathbb{F}_q}
\newcommand{\bbZp}[0]{\mathbb{Z}_p}
\newcommand{\bbQp}[0]{\mathbb{Q}_p}

\newcommand{\bbF}[0]{\mathbb{F}}
\newcommand{\bbL}[0]{\mathbb{L}}

\newcommand{\bbZ}[0]{\mathbb{Z}}
\newcommand{\bbA}[0]{\mathbb{A}}
\newcommand{\bbQ}[0]{\mathbb{Q}}

\newcommand{\SDVR}[0]{\mathbb{D}}          
\newcommand{\SSDVR}[0]{\mathcal{X}}          

\newcommand{\W}[0]{\mathbf{W}}          

\newcommand{\catvar}[0]{\mathsf{Var}}                

\newcommand\ili{\mathop{\hbox{\vtop{\setbox0=\hbox{lim}%
\dimen0\wd0\box0\vskip2pt
\nointerlineskip\hbox to \dimen0{\leftarrowfill}}}}\nolimits}
\newcommand\dli{\mathop{\hbox{\vtop{\setbox0=\hbox{lim}%
\dimen0\wd0\box0\vskip2pt
\nointerlineskip\hbox to\dimen0{\rightarrowfill}}}}\nolimits}

\DeclareMathOperator{\KO}{K_0}
\DeclareMathOperator{\ord}{ord}

\DeclareMathOperator{\spec}{Spec}

\newcommand{\muH}{\mu_{Haar}}

\DeclareMathOperator{\homology}{H}

\DeclareMathOperator{\symdiff}{\Delta}

\newcommand{\KOR}{\KO(\catvar_{\f})}
\newcommand{\KOL}{\mathcal{M}_{\f}}                                  
\newcommand{\KOW}{\overline{\KO}(\catvar_{\f})}
\newcommand{\KOWp}{\overline{\KO}(\catvar_{\bbFp})}

\newcommand{\KOC}{\widehat{\KOL}}
\newcommand{\KOT}{\overline{\KO}^{pol}(\catvar_{\f})}

\DeclareMathOperator{\counting}{C}
\DeclareMathOperator{\weight}{\overline{w}}
\DeclareMathOperator{\weighto}{w}

\DeclareMathOperator{\filtration}{F}

\numberwithin{equation}{section}

\newtheorem{lemma}[equation]{Lemma}

\newtheorem{proposition}[equation]{Proposition}
\newtheorem{example}[equation]{Example}
\newtheorem{definition}[equation]{Definition}
\newtheorem{remark}[equation]{Remark}
\newtheorem{definition-lemma}[equation]{Definition-Lemma}

\theoremstyle{remark}
\newtheorem*{acknowledgements}{Acknowledgment}

\begin{document}
\title[A version of geometric motivic integration]{A version of geometric motivic integration that specializes to $p$-adic integration via point counting}
\author{Karl R\"okaeus}
\address{Karl R\"okaeus \\ Department of Mathematics \\ Stockholm University \\ SE-106 91 Stockholm \\ Sweden}
\email{karlr@math.su.se}
\date{October 17, 2008}
\begin{abstract}
We give a version of geometric motivic integration that specializes to $p$-adic integration via point counting. This has been done before for stable sets, \emph{cf}. \cite{MR1997948}; we extend this to more general sets. The main problem in doing this is that it requires to take limits, hence the measure will have to take values in a completion of $\KO(\catvar_{\f})[\bbL^{-1}]$. The standard choice is to complete with respect to the dimension filtration; however, since the point counting homomorphism is not continuous with respect to this topology we have to use a stronger one. The first part of the paper is devoted to defining this topology; in the second part we will then see that many of the standard constructions of geometric motivic integration work also in this setting.
\end{abstract}

\maketitle

\section{Introduction}
There is a standard theory of geometric motivic integration, developed in \cite{MR1664700} and \cite{MR1905024}, and in \cite{MR2075915} for the case of mixed characteristic. In the present paper we give a version of that theory, valid over any complete discrete valuation ring, and with the property that it specializes to the Haar measure on $\bbZp^d$ in the case when we work over $\bbA_{\bbZp}^d$. 
The theory develops along the same lines as the standard theory of geometric motivic integration. The main difference is that the geometric motivic measure takes values in $\KOC$, i.e., $\KOL:=\KO(\catvar_{\f})[\bbL^{-1}]$ completed with respect to the dimension filtration; we instead complete $\KOL$ with respect to a stronger topology. The reason for this is that in the case when $\f$ is finite we want the point counting homomorphism to be continuous.

Let us quickly outline the theory: Let $\SSDVR$ be a variety defined over a complete discrete valuation ring, whose residue field is $\f$. As in all theories of geometric motivic integration, we begin by constructing the space of arcs on $\SSDVR$, denoted $\SSDVR_\infty$. There is a Boolean algebra of stable subsets of $\SSDVR_\infty$. On this algebra one may define a measure, taking values in $\KOL$, that specializes to the $p$-adic measure via point counting. This was done in \cite{MR1997948}. Next we construct a Boolean algebra of measurable subsets of $\SSDVR_\infty$, and a measure $\mu_\SSDVR$ on this algebra. The measure of a general measurable set is defined by covering it by stable sets and using a limiting process, also in the standard way. Since one needs to take limits to define the measure, it has to take values in a completion of $\KOL$. The standard choice is to use $\KOC$, the completion with respect to the dimension filtration. Since the point counting homomorphism is not continuous with respect to this filtration we have to use a stronger topology.

The construction of this topology is based on a previous construction, by Ekedahl \cite{Torsten}, of a topology on $\KOL$ with the property that, in case $\f$ is finite, taking the trace of Frobenius on the $l$-adic cohomology is continuous. By the Lefschetz trace formula, this topology then has the property we want, that the point counting homomorphism is continuous. However, this topology is probably too strong for our purposes, in that we are not able to make the definitions of geometric motivic integration work. We therefore have to modify it slightly, and in order to do this we introduce a partial ordering of $\KOL$. We use $\KOW$ to denote the completion of $\KOL$ with respect to the resulting topology. This topology is fine enough for, in case $\f$ is finite, the point counting homomorphism to be continuous, and still coarse enough for the standard definitions of geometric motivic integration to work.

The case that we are particularly interested in is when the discrete valuation ring is $\bbZp$, and $\SSDVR$ is an affine space over $\bbZp$. Let $\W(\overline{\bb{F}}_p)$ be the Witt vectors with coefficients in an algebraic closure of $\bbFp$. Suppose that $\SSDVR=\bbA_{\bbZp}^d$. Then $\SSDVR_\infty$ can be identified with $\W(\overline{\bb{F}}_p)^d$. Moreover, for every power of $p$, $q$, we have a homomorphism $\counting_q\colon\KOWp\to\bb{R}$ induced by counting $\bbFq$-points on $\bbFp$-varieties. The motivic measure has the property that it specialize to the (normalized) Haar measure, in the sense that for any measurable set $A\subset\SSDVR_\infty=\W(\overline{\bb{F}}_p)^d$ we have $\counting_p\mu_\SSDVR(A)=\muH(A\cap\bbZp^d)$. More generally, $\counting_{q}\mu_\SSDVR(A)=\muH(A\cap\W(\bbFq)^d)$ for any power $q$ of $p$. (Recall that $\bbZp\subset\W(\bb{F}_{p^f})\subset\W(\overline{\bb{F}}_p)$ is the integers in the unramified degree $f$ extension of $\bbQp$.) 

The main reason for us to develop this theory is that we have some explicit $p$-adic integrals that we want to compute in a motivic setting, in order to give an explanation of their behavior. These computations are carried out in a forthcoming paper, \cite{RokaeusComputations}


We now give an overview of the paper: In Section \ref{22} we define the completion of $\KOL$ that we work in, $\KOW$. A crucial property of this topology is that if $\f=\bbFp$ then the counting homomorphisms $\counting_{p^f}\colon\KOWp\to\bb{R}$ are continuous.

In Section \ref{23} we define the arc space and the motivic measure. The definitions given here are almost exactly the same as the ones used in the theory of geometric motivic integration, as presented for example in \cite{MR1886763} or the appendix of \cite{MR1905024}, and in \cite{MR2075915} for the mixed characteristic case.

It has to be remarked that this theory has some major shortcomings in its present state. Namely, when the variety has singularities, we are not able to prove that general cylinder sets are measurable (in particular, we are not able to prove that the arc space itself is measurable). That said, we are interested in utilizing the theory in the case when $\SSDVR=\bbA_{\bbZp}^d$, and in this case, and more generally for any smooth variety, everything works well.

\begin{acknowledgements}
The author is grateful to Professor Torsten Ekedahl for various suggestions and comments on the paper.
\end{acknowledgements}

\section{Background material on different completions of $\KOL$}\label{22}

We use $\KOL$ to denote $\KO(\catvar_{\f})[\bbL^{-1}]$. We use the standard definition of the dimension filtration of $\KOL$, namely for $m\in\bbZ$ define $\filtration^m\KOL$ to be the subgroup generated by $[X]/\bbL^r$, where $\dim X-r\leq m$. Define the \emph{dimension} of $x\in\KOL$ to be the minimal $m$ (possibly equal to $-\infty$) such that $x\in\filtration^m\KOL$. In the theory of geometric motivic integration one completes $\KOL$ with respect to this filtration to obtain $\KOC$. However, we are working from an arithmetic point of view; we want to define a theory of motivic integration for which the value of a motivic integral can be specialized to the corresponding $p$-adic one. For this reason it is natural to demand, in case $\f$ is finite, that the counting homomorphism is continuous. That this is not the case in the topology coming from the dimension filtration can be seen from the following simple example: Let $\f=\bbFq$ and consider the sequence $a_n=q^n/\bbL^n$. We have $\counting_q(a_n)=1$ for every $n$. However, $a_n\to0$ as $n\to\infty$, and $\counting_q(0)=0$. Because of this we are forced to complete with respect to a stronger topology. This topology should have the property that its definition is independent of the base field, and that if the base field is finite then the counting homomorphism is continuous.

Let us outline the contents of this section: In Subsection \ref{36} we define a notion of weights on $\KOL$. This definition is from \cite{Torsten}, where the author uses it to construct a topology on $\KOL$ which has the property that we want, that the counting homomorphism is continuous. However, this topology is a bit to strong when we define the motivic measure, so we have to modify it slightly. For that we use a partial ordering of $\KOL$, which we introduce in Subsection \ref{37}. Then finally in Subsection \ref{38} we define our topology, which is stronger than the filtration topology, but weaker than the topology from \cite{Torsten}. In Subsection \ref{39} we then prove some lemmas on the convergence of sums, which will be needed in later sections.

\subsection{Weights on $\KOL$}\label{36}
To define the topology we use a notion of weights of elements of $\KOL$, first given in \cite{Torsten}. We also refer to \cite{Torsten} for the proof that the following is well-defined. For $X$ a separated $\f$-scheme of finite type, we write $\homology_c(X)$ for the \'etale cohomology with compact support, with $\bbQ_l$-coefficients, of the extension of $X$ to a separable closure of $\f$.

\begin{definition}
We define the notion of weights of elements in $\KOL$.
\begin{itemize}
\item For a scheme $X$ of finite type over the base field $\f$, define for every integer $n$, $\weighto_n(X):=\sum_i\weighto_n\bigl(\homology^i_c(X)\bigr)$, where $\weighto_n\bigl(\homology^i_c(X)\bigr)$ is the dimension of the part of $\homology_c^i(X)$ of weight $n$.

\item For $x\in\KOR$, let $\weighto'_n(x)$ be the minimum of $\sum_i\abs{c_i}\weighto_n(X_i)$, where $\sum_ic_i[X_i]$ runs over all representations of $x$ as a linear combination of classes of schemes.

\item For $x\in\KOR$, define $\weight_n(x):=\lim_{i\to\infty}\weighto'_{n+2i}(x\bbL^i)$.

\item Finally, extend $\weight_n$ to $\KOL$ by $\weight_n(x/\bbL^i):=\weight_{n+2i}(x)$.
\end{itemize}
\end{definition}

\begin{example}\label{10}
$\homology_c^i(\bbA_{\f}^1)$ is equal to $\bbQ_l(-1)$, the cyclotomic representation, if $i=2$, and zero otherwise. From this one deduces that
$$\weight_n(\bbL^i)=\begin{cases}1\text{ if }n=2i\\0\text{ otherwise.}\end{cases}$$
\end{example}

The weight has the properties that it is subadditive, $\weight_n(x\pm y)\leq\weight_n(x)+\weight_n(y)$, and submultiplicative, $\weight_n(xy)\leq\sum_{i+j=n}\weight_i(x)\weight_j(y)$.

We next define the concept of uniform polynomial growth, introduced in \cite{Torsten}.
\begin{definition}
We say that a sequence of elements of $\KOL$, $(a_i)_{i\in\bbN}$, is of \emph{uniform polynomial growth} if there exists constants $l$, $C$ and $D$, independent of $i$ and  with the property that for every integer $n$, $\weight_n(a_i)\leq C\abs{n}^l+D$.
\end{definition}
One derives the following lemma from the facts that the the weight functions is subadditiv and submultiplicative. The details are in \cite{Torsten}.
\begin{lemma}\label{31}
Being of uniform polynomial growth is closed under termwise addition, subtraction and multiplication: If $(a_i)$ and $(b_i)$ are of uniform polynomial growth, then so are $(a_i\pm b_i)$ and $(a_ib_i)$.
\end{lemma}

\subsection{The partial ordering of $\KOL$}\label{37}

Before we can define our topology we also need a partial ordering on $\KOL$.
\begin{definition}
We introduce an ordering on $\KOL$ in the following way: First on $\KOR$ we define $x\leq y$ if there exists varieties $V_i$ such that $x+\sum[V_i]=y$. (Equivalently there exists a variety $V$ such that $x+[V]=y$.) We extend it to $\KOL$ by $x/\bbL^i\leq y/\bbL^j$ if there exists an $n$ such that $x\bbL^{j+n}\leq y\bbL^{i+n}$. (We see that $\bbL^{-1}>0$.)
\end{definition}
In particular, if $x$ is a linear combination of varieties, with positive coefficients, then $x>0$.

To work with this definition we need the following structure result about $\KOL$. The first part follows from \cite{Johannes}, Corollary $2.11$. The second part is proved similarly, alternatively, a proof is given in \cite{Edinburgh}, Example 2.6.
\begin{lemma}\label{ny1}
Let $n_i$ be non-negative integers and let $X_i$ be $\f$-varieties. Let $x=\sum_in_i[X_i]\in\KOL$. If $x=0$ then $n_i=0$ for every $i$. And if $x\in\filtration^m\KOL$, then $[X_i]\in\filtration^m\KOL$ for every $i$. 
\end{lemma}

\begin{lemma}
The ordering given above is a partial ordering of $\KOL$.
\end{lemma}
\begin{proof}
The only nontrivial thing to prove is antisymmetry. Suppose first that $x\leq y$ and $y\leq x$, where $x,y\in\KOR$. Then $x=y+\sum_i[U_i]$ and $y=x+\sum_i[V_i]$, giving together $\sum_i[V_i]+\sum_j[U_j]=0$. Lemma \ref{ny1} now shows that $[V_i]=[U_i]=0$ for every $i$.

If instead $x/\bbL^i\leq y/\bbL^j$ and $x/\bbL^i\geq y/\bbL^j$ in $\KOL$ then for some n, $\bbL^{j+n}x\leq\bbL^{i+n}y$ and $\bbL^{j+n}x\geq\bbL^{i+n}y$ in $\KOR$. It follows from the first part that $\bbL^{j+n}x=\bbL^{i+n}y\in\KOR$, hence that $x/\bbL^i= y/\bbL^j\in\KOL$.
\end{proof}
Given a variety $X$, every constructible subset of $U\subset X$ can be written as a finite disjoint union of locally closed subsets, $U=\bigcup_i U_i$. Since such a subset has a unique structure of a reduced subscheme we can take its class $[U_i]\in\KOR$, hence the class of $U$ is $[U]=\sum_i[U_i]\in\KOR$. Also, if $V\subset U$, then $V=\bigcup_i V_i$, where $V_i$ is locally closed and $V_i\subset U_i$. Here $V_i$ is closed in the subspace topology on $U_i$, hence $[U_i]-[V_i]=[U_i\setminus V_i]\geq0$. Consequently, $[U]-[V]=\sum_i[U_i\setminus V_i]\geq0$, which we state as the first part of the following lemma.
\begin{lemma}\label{4}
If $U$ and $V$ are constructible subsets of a variety, such that $V\subset U$, then $[V]\leq[U]$. Moreover, if $V\subset\bigcup_{i=1}^nU_i$ then $[V]\leq\sum_{i=1}^n[U_i]$.
\end{lemma}
\begin{proof}
For the second part, when  $n=2$ we have $[V]\leq[U_1\cup U_2]=[U_1]+[U_2]-[U_1\cap U_2]\leq[U_1]+[U_2]$. The general statement follows by induction on $n$.
\end{proof}

\begin{lemma}\label{35}
Let $x,a,b\in\KOL$. If $a\leq x\leq b$ then $\dim x\leq\max\{\dim a,\dim b\}$.
\end{lemma}
\begin{proof}
There are varieties $X$ and $Y$ such that $x=a+[X]$ and $b=x+[Y]$. This shows that $b-a=[X]+[Y]$. By Lemma \ref{ny1}, $\dim[X]\leq \dim(b-a)$. Hence $\dim x=\dim(a+[X])\leq\max\{\dim a,\dim[X]\}\leq\max\{\dim a,\dim(b-a)\}\leq\max\{\dim a,\dim b\}$.
\end{proof}

\subsection{The topology on $\KOL$}\label{38}

We now define our topology on $\KOL$, by specifying what it means for a sequence to converge. Since we work in a group it suffices to tell what it means for a sequence to converge to zero.
\begin{definition}
Let $(x_i)$ be a sequence of elements in $\KOL$.
\begin{itemize}
\item We say that $x_i$ is \emph{strongly convergent to 0} if it is of uniform polynomial growth and $\dim x_i\to-\infty$.
\item $x_i$ converges to zero, $x_i\to0$, if there are sequences $a_i$ and $b_i$ that converges to zero strongly, and such that $a_i\leq x_i\leq b_i$.
\item $(x_i)$ is \emph{Cauchy} if $x_i-x_j\to0$ when $i,j\to\infty$.
\end{itemize}
We define a topology on $\KOL$ by stipulating that a subset is closed if it contains all its limit points with respect to this notion of convergence.
\end{definition}
One proves that this is a topological ring. It then follows immeditly that if a sequence of elements in $\KOL$ is convergent then it is Cauchy. Moreover, we have the property that if $x_i\leq y_i\leq z_i$ and if $x_i$ and $z_i$ tends to zero, then $y_i$ tends to zero, a fact that we will use without further notice.

We compare this to the standard topology on $\KOL$:
\begin{lemma}
If a sequence is convergent (respectively Cauchy), then it is convergent (respectively Cauchy) with respect to the dimension filtration. In particular, our topology is stronger than the dimension topology.
\end{lemma}
\begin{proof}
Let the sequence be $x_i$ be convergent to zero. There are sequences $a_i$ and $b_i$ strongly convergent to zero and such that $a_i\leq x_i\leq b_i$. By Lemma \ref{35}, $\dim x_i\leq\max\{\dim a_i,\dim b_i\}$ and it follows that $\dim x_i\to-\infty$.
\end{proof}
In particular, if we know that a sequence $x_n$ is convergent then $x_n\to0$ if and only if $\dim x_n\to-\infty$.

We are now ready to give the completion that let the motivic integral specialize to the corresponding $p$-adic integral. We define $\KOW$, the completion of $\KOL$, in the following way: Consider the set of Cauchy sequences in $\KOL$. This is a ring under termwise addition and multiplication, which one proves using Lemma \ref{31}. Moreover it has a subset consisting of those sequences that converges to zero, and it is straight forward to prove that this subset is an ideal. We define $\KOW$ to be the quotient by this ideal. Moreover we have a completion map $\KOL\to\KOW$ that takes $x$ to the image of the constant sequence $(x)$. We state this as a formal definition:
\begin{definition}
$\KOW$ is the ring of Cauchy sequences modulo the ideal of those sequences that converges to zero.
\end{definition}
For any sequence $(x_i)$ of elements of $\KOL$ which is Cauchy with respect to dimension, $\dim x_i$ is eventually constant, or it converges to $-\infty$. Also, given $n$, $\weight_n(x_i)$ is eventually constant. If $x=(x_i)\in\KOW$, define $\dim x$ and $\weight_n(x)$ to be these constant values. These functions keep there basic properties, e.g., subadditivity for $\weight_n$. We may then extend the concept of being of uniform polynomial growth to $\KOW$. Moreover, we let $(x_i)\leq(y_i)$ if for some $i_0$, $x_i\leq y_i$ for every $i\geq i_0$. We define the topology on $\KOW$ in the same way as on $\KOL$, by saying that a sequence converges to zero if and only if it is bounded from above and below by sequences strongly convergent to zero. The following result justifies the fact that we refer to $\KOW$ as the completion of $\KOL$.

\begin{lemma}
The completion map $\KOL\to\KOW$ is continuous and $\KOW$ is complete, in the sense that any Cauchy sequence converges. Moreover, the image of $\KOL$ is dense in $\KOW$
\end{lemma}
\begin{proof}
We prove that for any $x=(x_j)_{j\in\bbN}\in\KOW$ we have $x_i\to x$ as $i\to\infty$. This will show that any Cauchy sequence of elements in $\KOL$ converges and that $\KOL$ is dense in $\KOW$. For this we have to prove that $y_i:=x_i-x=(x_i-x_j)_{j\in\bbN}\to0$ as $i\to\infty$. 

There are two sequences $a_{ij}$ and $b_{ij}$ which are strongly convergent to zero as $i,j\to\infty$, and such that $a_{ij}\leq x_i-x_j\leq b_{ij}$. Hence $(a_{ij})_{j\in\bbN}\leq y_i\leq(b_{ij})_{j\in\bbN}$. We prove that $(a_{ij})_j$ tends strongly to zero as $i\to\infty$: The dimension of $(a_{ij})_j$ is $\dim a_{i,f(i)}$, where $f(i)$ is some sufficiently large integer which we may and will assume is greater than $i$. Moreover, $\weight_n((a_{ij})_j)=\weight_n(a_{i,g(i)})$ for some $g(i)$. Since $a_{ij}$ is strongly convergent, $\dim a_{i,f(i)}\to-\infty$ as $i\to\infty$, and $\weight_n(a_{i,g(i)})\leq C\abs{n}^l+D$ for every $i$.
\end{proof}

Let $\KOC$ be the completion of $\KOL$ with respect to the dimension filtration. We have an injective, continuous homomorphism $\KOW\to\KOC$ so we may think of $\KOW$ as a subring of $\KOC$, although its topology is stronger than the subspace topology.

\begin{example}\label{34}
By Example \ref{10} and subadditivity, $\weight_n(\sum_{m=i}^j\bbL^{-m})\leq1$ for every $i,j$. Furthermore $\dim\bbL^{-m}\to-\infty$ as $m\to\infty$. Together this shows that the sequence $(\sum_{m=0}^N\bbL^{-m})_N$ is Cauchy, hence that the sum $\sum_{m\in\bbN}\bbL^{-m}$ is convergent in $\KOW$. In the same way one sees that $\bbL^{-m}$ converges (strongly) to zero, hence letting $N$ tend to infinity in the equality  $(1-\bbL^{-1})(\sum_{m=0}^N\bbL^{-m})=1-\bbL^{-(N+1)}$ shows that $\sum_{m\in\bbN}\bbL^{-m}=1/(1-\bbL^{-1})\in\KOW.$ In particular, since also $\bbL$ is invertible, it follows that $\bbL-1$ is invertible.

Similarly one proves that if $\{e_i\}_{i\in\bbN}$ is a sequence of integers such that $e_i\to\infty$, then $\sum_{i\in\bbN}\bbL^{-e_i}$ is convergent.
\end{example}

\begin{remark}
The main reason for us to use this completion is that it makes the point counting homomorphism well behaved. In \cite{Torsten} the author uses an even stronger topology, using only sequences strongly convergent, to obtain $\KOT$. In $\KOT$ the point counting homomorphism is continuous. However, the topology of $\KOT$ is little bit to strong for our purposes; it does not allow us to define the motivic measure (Definition \ref{32}). However, in $\KOT$ point counting can be made by computing the trace of Frobenius; the Euler characteristic is continuous, a property that we have to sacrifice in $\KOW$. For a more thorough discussion of this, see \cite{Edinburgh}.
\end{remark}

When $\f$ is finite, this topology makes the point counting homomorphism continuous:

\begin{definition-lemma}
For $\f=\bbFq$, define  $\counting_q\colon\KOR\to\bbZ$ by $[X]\mapsto\abs{X(\f)}$. It is a ring homomorphism and it extends to a homomorphism $\KOL\to\bbQ$, continuous with respect to the above constructed topology. It hence extends by continuity to a continuous homomorphism $\counting_q\colon\KOW\to\bb{R}$.
\end{definition-lemma}
\begin{proof}
In \cite{Torsten} it is proved that $\counting_q$ is continuous with respect to the topology of $\KOT$, i.e., if $a_n\to0$ strongly then $\counting_qa_n\to0$. Now, if $x_n\to0$ then there are sequences $a_n$ and $b_n$ such that $a_n\leq x_n\leq b_n$ and $a_n,b_n\to0$ strongly. Since $\counting_qa_n\leq\counting_qx_n\leq\counting_qb_n$ it follows that $\counting_qx_n\to0$, consequently $\counting_q$ is continuous.
\end{proof}

\subsection{Various lemmas on the convergence of series}\label{39}

We collect here some lemmas that will be needed when working with this definition.
\begin{lemma}\label{33}
Suppose that $x_i\leq y_i\leq z_i$. If $\sum_{i\in\bbN}x_i$ and $\sum_{i\in\bbN}z_i$ are convergent, then so is $\sum_{i\in\bbN}y_i$. Moreover, $\sum_{i\in\bbN}x_i\leq\sum_{i\in\bbN}y_i\leq\sum_{i\in\bbN}z_i$.
\end{lemma}
\begin{proof}
We have $\sum_{i=m}^nx_i\leq\sum_{i=m}^ny_i\leq\sum_{i=m}^nz_i$, and since $\sum_{i=m}^nx_i\to0$ and $\sum_{i=m}^nz_i\to0$ as $m,n\to\infty$, it follows that $\sum_{i=m}^ny_i$ is bounded from above and below by sequences strongly convergent to zero, hence it converges to zero. So $(\sum_{i=0}^Ny_i)_N$ is Cauchy, the sum is convergent. The second assertion follows by definition, since it holds for each partial sum.
\end{proof}
Note that if $x_i\leq y_i\leq z_i$ it does not follow that $y_i$ is convergent when $x_i$ and $z_i$ are.

Let $a_i\in\KOW$. In general it is not true that $\sum a_i$ is convergent if and only if $a_i\to0$, a property that holds for $\KOC$. However, some of the consequences of this are true in a special case:
\begin{lemma}
If $a_i\geq0$ for every $i$ and if $\sum_ia_i$ is convergent then every rearrangement of the sum is convergent, and to the same limit.
\end{lemma}
\begin{proof}
If $(b_i)_{i\in\bbN}$ is a rearrangement of $(a_i)_{i\in\bbN}$ then, for some $N_n$, the elements $b_0\dotsk b_n$ are among $a_0\dotsk a_{N_n}$. Therefore $\sum_{i\leq N_n}(b_i-a_i)$ is an alternating sum of $a_j$, with $j>n$. Since every $a_i\geq0$, this sum is between $-\sum_{j=n+1}^Ma_i$ and  $\sum_{j=n+1}^Ma_i$ for some $M$. Since $\sum_{i\in\bbN}a_i$ is Cauchy, both these sums tend to $0$ as $n$ tends to $\infty$. Hence $\sum_{i\in\bbN}(b_i-a_i)=0$.
\end{proof}




For this reason, we will write $\sum_{i,j\in\bbN}a_{ij}$ to mean the sum over some unspecified enumeration of $\bbN^2$.
\begin{lemma}\label{9}
Assume that all $a_{ij}\geq0$. If the sum $\sum_{i,j\in\bbN}a_{ij}$ is convergent then the same holds for $\sum_{i\in\bbN}\sum_{j\in\bbN}a_{ij}$, and the two sums are equal.
\end{lemma}
\begin{proof}
For every $i\in\bbN$, it follows from Lemma \ref{33}, and the convergence of $\sum_{i,j\in\bbN}a_{ij}$, that $\sum_{j\in\bbN}a_{ij}$ is convergent. We then have
\begin{align*}
\sum_{i\leq n}\sum_{j\in\bbN}a_{ij}-\sum_{i,j\leq n}a_{ij}&=\sum_{i\leq n}\sum_{j>n}a_{ij}\\
&=\sum_{j>n}\sum_{i\leq n}a_{ij},
\end{align*}
Now $\sum_{j\in\bbN}(\sum_{i=0}^ja_{ij})$ is convergent, because of Lemma \ref{33} and since every rearrangement of $\sum_{i,j\in\bbN}a_{ij}$ is. Hence $\lim_{n\to\infty}\sum_{j>n}(\sum_{i\leq n}a_{ij})=0$ (because if $\sum_{i\in\bbN}b_i=b$ then $\sum_{i>n}b_i=b-\sum_{i\leq n}b_i\to0$). The result follows.
\end{proof}





\begin{lemma}\label{14}
Suppose that $a_i\geq0$, $b_i\geq0$. The sum $\sum_{n\in\bbN}\sum_{i+j=n}a_ib_j$ is convergent if and only if $\sum_{i\in\bbN}a_i\sum_{i\in\bbN}b_i$ is. In this case they are equal.
\end{lemma}
\begin{proof}
$$\Biggl(\sum_{0\leq i\leq N}a_i\Biggr)\Biggl(\sum_{0\leq i\leq N}b_i\Biggr)-\sum_{0\leq n<N}\sum_{i+j=n}a_ib_j=\sum_{I(N)}a_ib_j,$$
where $I(N)$ is a set of indices $(i,j)$, all of which $i+j\geq N$. Now, since every $a_i$ and $b_i$ is $\geq0$, it follows that $0\leq\sum_{I(N)}a_ib_j\leq\sum_{n\geq N}\sum_{i+j=n}a_ib_j$ for every $N$. Since $\sum_{n\in\bbN}(\sum_{i+j=n}a_ib_j)$ is Cauchy it follows that $\sum_{n\geq N}\sum_{i+j=n}a_ib_j\to0$ as $N\to\infty$, hence $\sum_{I(N)}a_ib_j\to0$.
\end{proof}

\begin{lemma}\label{12}
If $a_i\geq0$, $b_i\geq0$, and if $\sum_{i\in\bbN}a_i$ and $\sum_{i\in\bbN}b_i$ are convergent, then $\sum_{i\in\bbN}a_ib_i$ is convergent and $\leq\sum_{i\in\bbN}a_i\sum_{i\in\bbN}b_i$.
\end{lemma}
\begin{proof}
Lemma \ref{14} shows that $\sum_{n\in\bbN}(\sum_{i+j=n}a_ib_j)$ is convergent. Since $a_nb_n\leq\sum_{i+j=2n}a_ib_j$, and $0\leq\sum_{i+j=2n+1}a_ib_j$ it hence follows from Lemma \ref{33} that $\sum_{i\in\bbN}a_ib_i$ is convergent, and less than or equal to $\sum_{n\in\bbN}(\sum_{i+j=n}a_ib_j)$.
\end{proof}

\section{Definition of the motivic measure}\label{23}
Before we start, let us remark that all the definitions in this section are the standard ones, used in the appendix of \cite{MR1905024}, in \cite{MR1886763} and in \cite{MR2075915}. (There are some minor differences between these expositions; in these cases we have followed the path of \cite{MR1886763}. In particular, our arc spaces will be sets, rather than schemes or formal schemes.) However, since we use a stronger topology, the proofs that everything works is slightly different, even though they are based on the corresponding existing proofs. In the end of the section we also prove that we now have the property that we want, that we may specialize to $p$-adic integration via point counting.

Given a complete discrete valuation ring, let $\SDVR$ be its spectrum and $\f$ its residue field. Given a scheme $\SSDVR$ over $\SDVR$, of finite type and of pure relative dimension $d$, define $\SSDVR_n$ to be the $n$th Greenberg variety over $\f$. In the mixed characteristic case, $\SSDVR_n$ is characterized by the property that if $R$ is a $\f$-algebra then $\SSDVR_n(R)=\SSDVR\bigl(\W_{n}(R)\bigr)$, where $\W$ is the Witt vectors. In the equal characteristic case, $\SSDVR_n$ is characterized by the property that if $R$ is a $\f$-algebra then $\SSDVR_n(R)=\SSDVR(R[T]/(t^{n})\bigr)$. Fix an algebraic closure $\overline{\f}$ of $\f$ and define $\SSDVR_\infty$ to be the projective limit of the sets $\SSDVR_n(\overline{\f})$. For  $\f\subset K\subset\overline{\f}$, define $\SSDVR_\infty(K)$ to be $\ili\SSDVR_n(K)\subset\SSDVR_\infty$. We are going to define the measure of certain subsets of $\SSDVR_\infty$.

We say that a subset $S\subset\SSDVR_n(\overline{\f})$ is \emph{constructible} if there are a finite number of locally closed subschemes $V_i\subset\SSDVR_n$ such that $S=\bigcup_iV_i(\overline{\f})$, where the $V_i$ are mutually disjoint.

If $A\subset\SSDVR_\infty$, then it is a \emph{cylinder} if the following holds: For some $n$, $\pi_n(A)$ is constructible in $\SSDVR_n(\overline{\f})$, and defined over $\f$, and $A=\pi_n^{-1}\pi_n(A)$. In this case there are subschemes $V_i\subset\SSDVR_n$ such that $\pi_n(A)=\bigcup_i V_i(\overline{\f})$; we define $[\pi_n(A)]:=\sum_i[V_i]\in\KOR$.

If, in addition, for all $m\geq n$, the projection $\pi_{m+1}(A)\to\pi_m(A)$ is a piecewise trivial fibration, with fiber an affine space of dimension $d$, then we say that $A$ is a \emph{stable cylinder}, or just that it is \emph{stable}. (Every cylinder is stable in case $\SSDVR$ is smooth.)

If $A$ is stable we see that $\dim\pi_m(A)-md$ is independent of the choice of $m\geq n$, define the \emph{dimension} of $A$ to be this number.

We also see that $\tilde{\mu}_{\SSDVR}(A):=[\pi_m(A)]\bbL^{-md}\in\KOL$ is independent of the choice of $m\geq n$. This is an additive measure on the set of stable subsets. We want to define a measure on a bigger collection of subsets of $\SSDVR_\infty$. We will use the standard construction, except that we let the measure take values in $\KOW$ instead of $\KOC$. So let $\mu_{\SSDVR}(A)$, the measure of $A$, be the image of $\tilde{\mu}_{\SSDVR}(A)$ in $\KOW$. Note that $\dim\mu_\SSDVR(A)=\dim A$.

\begin{example}
The case that we are particularly interested in is when $\SDVR=\spec\bbZp$ (so $\f=\bbFp$) and $\SSDVR=\bbA^d_{\bbZp}$. Then
$$\SSDVR_\infty=\lim_{\longleftarrow}\bbA_{\bbZp}^d\bigl(\W_n(\overline{\f})\bigr)=\W(\overline{\f})^d.$$
We see that $\SSDVR_\infty(\f)=\bbZp^d$. By the standard construction of the (normalized) Haar measure on $\bbZp^d$ we see that if $A$ is stable then $\mu_{Haar}(A\cap\bbZp^d)=\counting_p(\mu_{\SSDVR}(A))$ (\emph{cf}. \cite{MR1997948}, Lemma 4.6.2). Below we define a concept of measurability of subsets of $\SSDVR_\infty$, for general $\SSDVR$, in such a way that in the special case when $\SSDVR=\bbA^d_{\bbZp}$, if $A\subset\SSDVR_\infty$ is measurable then $\mu_{Haar}(A\cap\bbZp^d)=\counting_p(\mu_{\SSDVR}(A))$.
\end{example}

We sum up the above observations in the following lemma:
\begin{lemma}\label{1}
The set of stable subsets of $\SSDVR_\infty$ form a Boolean ring
on which $\mu_\SSDVR$ is additive. Moreover, if $\SSDVR=\bbA^d_{\bbZp}$ and $A\subset\SSDVR_\infty$ is stable, then $\mu_{Haar}(A\cap\bbZp^d)=\counting_p(\mu_{\SSDVR}(A))$.
\end{lemma}

We next extend the motivic measure to a bigger collection of subsets: We want to do this in a way similar to that of the ordinary Haar measure. The problem is that the standard construction involves taking $\sup$ or $\inf$, which we cannot do in $\KOW$. We therefore use the same method as is used in \cite{MR1886763}, except that we use $\KOW$ instead of $\KOC$: 
\begin{definition}\label{32}
The subset $A\subset\SSDVR_\infty$ is \emph{measurable} if the following holds:
\begin{itemize}
\item For every positive integer $m$, there exists a stable subset $A_m\subset\SSDVR_\infty$ and a sequence of stable subsets $(C_i^m\subset\SSDVR_\infty)_{i=1}^\infty$, with $\sum_i\mu_{\SSDVR}(C_i^m)$ convergent, such that $A\symdiff A_m\subset\bigcup_i C_i^m$.

\item $\lim_{m\to\infty}\sum_i\mu_{\SSDVR}(C_i^m)=0$.
\end{itemize}
The measure of $A$ is then defined to be $\mu_{\SSDVR}(A):=\lim_{m\to\infty}\mu_{\SSDVR}(A_m)$.
\end{definition}
\begin{remark}
This definition works well when $\SSDVR$ is smooth, since in that case all cylinders are stable. In the general case, it is probably better to first define the measure of an arbitrary cylinder $A$ as $\lim_{e\to\infty}\mu_{\SSDVR}(\SSDVR_\infty^{(e)}\cap A)\in\KOW$, where $\SSDVR_\infty^{(e)}=\SSDVR_\infty\setminus\pi_e^{-1}\bigl((\SSDVR_{sing})_e\bigr)$, and then replace ``stable sets'' with ``cylinders'' in the definition of a measurable set. However, since we are not able to prove that this limit exists in general, we use the present definition, which in any case works for our purposes.
\end{remark}

To prove that this is well-defined we will use the following well known lemma (\emph{cf.} Lemma $3.5.1$ in  \cite{MR1997948}).
\begin{lemma}\label{3}
Let  $A\subset\SSDVR_\infty$ be stable. Then every countable covering of $A$ with stable subsets $C_i$ has a finite subcovering.
\end{lemma}

\begin{proposition}\label{2}
The notion of measureability of subsets of $\SSDVR_\infty$ is well-defined. The measurable subsets form a Boolean ring on which $\mu_\SSDVR$ is additive.
\end{proposition}
\begin{proof}
We first prove that the limit exists. For this we have to prove that the sequence $(\mu_\SSDVR(A_m))_{m=1}^{\infty}$ is Cauchy. Since $A_m$ and $A_{m'}$ are stable, $\mu_\SSDVR(A_m\symdiff A_{m'})=\mu_\SSDVR(A_m)-\mu_\SSDVR(A_{m'})+2\mu_\SSDVR(A_{m'}\setminus A_m)$. Moreover, $\mu_\SSDVR(A_{m'}\setminus A_m)\geq0$ so it follows that $\mu_\SSDVR(A_m)-\mu_\SSDVR(A_{m'})\leq\mu_\SSDVR(A_m\symdiff A_{m'})$. Next, $A_m\symdiff A_{m'}\subset\bigcup_i C_i^m\cup C_i^{m'}$. By Lemma \ref{3} a finite number of the $C_i^m\cup C_i^{m'}$ suffices, hence the right hand side is stable so $\mu_\SSDVR(A_m\symdiff A_{m'})\leq\mu_\SSDVR(\bigcup_{i=0}^N C_i^m\cup C_i^{m'})$. From Lemma \ref{4} we then see that $\mu_\SSDVR(A_m\symdiff A_{m'})\leq\sum_{i=0}^N\bigl(\mu_\SSDVR(C_i^m)+\mu_\SSDVR(C_i^{m'})\bigr)$. Hence
$$\mu_\SSDVR(A_m)-\mu_\SSDVR(A_{m'})\leq\sum_{i=0}^N\bigl(\mu_\SSDVR(C_i^m)+\mu_\SSDVR(C_i^{m'})\bigr).$$
Similarly, $-\sum_{i=0}^N\bigl(\mu_\SSDVR(C_i^m)+\mu_\SSDVR(C_i^{m'})\bigr)\leq\mu_\SSDVR(A_m)-\mu_\SSDVR(A_{m'})$. Since, by assumption, $\sum_{i=0}^N\mu_\SSDVR(C_i^m)$ converges to zero as $m$ tends to infinity, it follows that $\mu_\SSDVR(A_m)-\mu_\SSDVR(A_{m'})$ is bounded from above and below by sequences strongly convergent to zero.

Next, suppose that $A\symdiff B_m\subset\bigcup_iD_i^m$ is another sequence that defines $\mu_\SSDVR(A)$. In the same way as above we see that for some $N$, $-\sum_{i=0}^N\bigl(\mu_\SSDVR(C_i^m)+\mu_\SSDVR(D_i^m)\bigr)\leq\mu_\SSDVR(A_m)-\mu_\SSDVR(B_m)\leq\sum_{i=0}^N\bigl(\mu_\SSDVR(C_i^m)+\mu_\SSDVR(D_i^m)\bigr)$, hence $\mu_\SSDVR(A_m)-\mu_\SSDVR(B_m)\to0$ as $m\to\infty$.
\end{proof}

The following proposition is sometimes useful when proving that a set is measurable.
\begin{proposition}\label{5}
Let $A=\bigcup_{i\in\bbN}A_i$, where the $A_i$ are stable and the sum $\sum_{i\in\bbN}\mu_\SSDVR(A_i)$ is convergent. Then
$A$ is measurable and $\mu_\SSDVR(A)=\lim_{n\to\infty}\mu_\SSDVR(\bigcup_{i\leq n}A_i)$. If furthermore the $A_i$ are pairwise disjoint, then $\mu_\SSDVR(A)=\sum_{i\in\bbN}\mu_\SSDVR(A_i)$.
\end{proposition}
\begin{proof}
We have $A\symdiff\bigcup_{i\leq n}A_i=\bigcup_{i>n}A_i$. Here $\sum_{i>n}\mu_\SSDVR(A_i)$ is convergent since $\sum_{i\in\bbN}\mu_\SSDVR(A_i)$ is, and it then follows that $\lim_{n\to\infty}\sum_{i>n}\mu_\SSDVR(A_i)=\lim_{n\to\infty}\bigl(\sum_{i\in\bbN}\mu_\SSDVR(A_i)-\sum_{i\leq n}\mu_\SSDVR(A_i)\bigr)=\sum_{i\in\bbN}\mu_\SSDVR(A_i)-\lim_{n\to\infty}\sum_{i\leq n}\mu_\SSDVR(A_i)=0$. Therefore, by definition, $A$ is measurable and $\mu_\SSDVR(A)=\lim_{n\to\infty}\mu_\SSDVR(\bigcup_{i\leq n}A_i)$. In the case when the $A_i$ are disjoint we have $\mu_\SSDVR(\bigcup_{i\leq n}A_i)=\sum_{i\leq n}\mu_\SSDVR(A_i)$, hence $\mu_\SSDVR(A)=\sum_{i\in\bbN}\mu_\SSDVR(A_i)$.
\end{proof}
We will use the following lemma to apply the proposition.
\begin{lemma}\label{29}
Suppose that $A_i\subset B_i\subset\SSDVR_\infty$ are stable. If $\sum_{i\in\bbN}\mu_\SSDVR(B_i)$ is convergent, then so is $\sum_{i\in\bbN}\mu_\SSDVR(A_i)$.
\end{lemma}
\begin{proof}
For some $n$, both $A_i$ and $B_i$ are stable of level $n$. We have $\pi_n(A_i)\subset\pi_n(B_i)$, hence by Lemma \ref{4}, $[\pi_n(A_i)]\leq[\pi_n(B_i)]$. Therefore $0\leq\mu_\SSDVR(A_i)\leq\mu_\SSDVR(B_i)$, it follows from Lemma \ref{33} that the sum is convergent.
\end{proof}

The following proposition shows that this definition generalizes the $p$-adic measure.
\begin{proposition}
Let $\SSDVR=\bbA_{\bbZp}^d$. When $A\subset\SSDVR_\infty$ is measurable we have $\mu_{Haar}(A\cap\bbZp^d)=\counting_p(\mu_{\SSDVR}(A))$.
\end{proposition}
\begin{proof}
By definition, we have that $\counting_p\mu_{\SSDVR}(A)=\counting_p\lim_{m\to\infty}\mu_{\SSDVR}(A_m)$. Since $\counting_p$ is continuous, this equals $\lim_{m\to\infty}\counting_p\mu_{\SSDVR}(A_m)=\lim_{m\to\infty}\muH(A_m\cap\bbZp^d)$. If we now can show that $\muH\bigl((A\symdiff A_m)\cap\bbZp^d\bigr)\rightarrow0$ when $m\to\infty$, then by standard measure theory it follows that $\lim_{m\to\infty}\muH(A_m\cap\bbZp^d)=\muH(A\cap\bbZp^d)$, and we are done.

Since $(A\symdiff A_m)\cap\bbZp^d\subset(\bigcup_i C_i^m)\cap\bbZp^d$ it suffices to show that $\lim_{m\to\infty}\sum_i\muH(C_i^m\cap\bbZp^d)=0$. Now this equals $\lim_{m\to\infty}\sum_i\counting_p\mu_\SSDVR(C_i^m)$. Using two times the fact that $\counting_p$ is continuous, together with the assumption on the $C_i$ gives that this equals $\counting_p\lim_{m\to\infty}\sum_i\mu_\SSDVR(C_i^m)=\counting_p0=0$
\end{proof}

We next define the integrals of functions $\SSDVR_\infty\to\KOW$.
\begin{definition}
Let the function $f\colon\SSDVR_\infty\to\KOW$ have measurable fibers and the property that the sum $\sum_{a\in\KOW}\mu_\SSDVR(f^{-1}(a))a$ is convergent, (in particular it has only countably many nonzero terms). We then say that $f$ is integrable and we define $\int fd\mu_\SSDVR$ to be the above limit.
\end{definition}

If $A\subset\SSDVR_\infty$, let $\chi_A$ be the characteristic function of $A$ and define $\int_Afd\mu_\SSDVR:=\int f\cdot\chi_Ad\mu_\SSDVR$. If $\int_Afd\mu_\SSDVR$ and $\int_Bfd\mu_\SSDVR$ exists and $A$ and $B$ are disjoint then $\int_{A\cup B}fd\mu_\SSDVR$ exists and is equal to $\int_Afd\mu_\SSDVR+\int_Bfd\mu_\SSDVR$ (By the definition, since $\lim x_n+\lim y_n=\lim(x_n+y_n)$ holds by definition of addition in ring of Cauchy sequences.)

\begin{proposition}\label{16}
Let $\SSDVR=\bbA_{\bbZp}^d$. If $A\subset\SSDVR_\infty$ is measurable and if $f\colon\SSDVR_\infty\to\KOW$ is integrable, then
$$\counting_p\int_A fd\mu_\SSDVR=\int_{A\cap\bbZp^d}\counting_p\circ f d\muH.$$
\end{proposition}
\begin{proof}
As we have set things up, this is straight forward:
\begin{align*}
\counting_p\int_Afd\mu_\SSDVR=&\counting_p\sum_{a\in\KOW}\mu_\SSDVR(f^{-1}a\cap A)a\\
=&\sum_{a\in\KOW}\muH(f^{-1}a\cap A\cap\bbZp^d)\counting_pa\\
=&\sum_{r\in\bb{R}}\muH\bigl((\counting_p\circ f)^{-1}r\cap A\cap\bbZp^d\bigr)r\\
=&\int_{A\cap\bbZp^d}\counting_p\circ f d\muH.  \qedhere
\end{align*}
\end{proof}

We illustrate this with an example: 
\begin{example}\label{43}
Let $\SSDVR=\bbA^1_{\bbZp}$. Define the function $\abs{\cdot}\colon\SSDVR_\infty\to\KOWp$ as $x\mapsto\bbL^{-\ord x}$. In a forthcoming paper \cite{RokaeusComputations} we will prove that $$\int_{\SSDVR_\infty}\abs{X^2+1}d\mu_\SSDVR=1-[\spec\bbFp[X]/(X^2+1)]\tfrac{1}{\bbL+1}\in\KOWp.$$
This integral has the property that for every power of $p$, $\counting_q\int_{\SSDVR_\infty}\abs{X^2+1}d\mu_\SSDVR=\int_{\W(\bbFq)}\abs{X^2+1}_pdX$. So by computing the motivic integral we have simultaneously computed the corresponding integral over $\W(\bbFq)$, for every power of $p$. In this example, if $p\equiv 1\pmod{4}$ then $-1$ is a square in $\bbFp$ and hence $\bbFp[X]/(X^2+1)=\bbFp^2$. So  $\int_{\SSDVR_\infty}\abs{X^2+1}d\mu_\SSDVR=1-\tfrac{2}{\bbL+1}$, showing that for every power of $p$ we have $\int_{\W(\bbF_{q})}\abs{X^2+1}_pdX=(q-1)/(q+1)$.

If instead $p\equiv 3\pmod{4}$ then $-1$ is a non-square in $\bbFq$ and $\bbFp[X]/(X^2+1)=\bb{F}_{p^2}$. So $\int_{\SSDVR_\infty}\abs{X^2+1}d\mu_\SSDVR=1-[\spec\bbF_{p^2}]\tfrac{1}{\bbL+1}$, showing that
$$\int_{\W(\bbF_{q})}\abs{X^2+1}_pdX=\begin{cases}
                                      1 &\text{ if $q=p^{2m+1}$}\\
                                      (q-1)/(q+1) &\text{ if $q=p^{2m}$}
                                     \end{cases}
$$
\end{example}

\bibliography{MF-bibliography}

\providecommand{\bysame}{\leavevmode\hbox to3em{\hrulefill}\thinspace}
\providecommand{\MR}{\relax\ifhmode\unskip\space\fi MR }
\providecommand{\MRhref}[2]{%
  \href{http://www.ams.org/mathscinet-getitem?mr=#1}{#2}
}
\providecommand{\href}[2]{#2}
\begin{thebibliography}{R{\"o}k08b}

\bibitem[DL99]{MR1664700}
Jan Denef and Fran{\c{c}}ois Loeser, \emph{Germs of arcs on singular algebraic
  varieties and motivic integration}, Invent. Math. \textbf{135} (1999), no.~1,
  201--232. \MR{MR1664700 (99k:14002)}

\bibitem[DL02]{MR1905024}
\bysame, \emph{Motivic integration, quotient singularities and the {M}c{K}ay
  correspondence}, Compositio Math. \textbf{131} (2002), no.~3, 267--290.
  \MR{MR1905024 (2004e:14010)}

\bibitem[Eke07]{Torsten}
Torsten Ekedahl, \emph{On the class of an algebraic stack},
  www.mittag-leffler.se/preprints/0607/info.php?id=66 (2007).

\bibitem[Loo02]{MR1886763}
Eduard Looijenga, \emph{Motivic measures}, Ast\'erisque (2002), no.~276,
  267--297, S\'eminaire Bourbaki, Vol.\ 1999/2000. \MR{MR1886763 (2003k:14010)}

\bibitem[LS03]{MR1997948}
Fran{\c{c}}ois Loeser and Julien Sebag, \emph{Motivic integration on smooth
  rigid varieties and invariants of degenerations}, Duke Math. J. \textbf{119}
  (2003), no.~2, 315--344. \MR{MR1997948 (2004g:14026)}

\bibitem[Nic08]{Johannes}
Johannes Nicaise, \emph{A trace formula for varieties over a discretely valued
  field}, arxiv.org/abs/0805.1323 (2008).

\bibitem[R{\"o}k08a]{RokaeusComputations}
Karl R{\"o}kaeus, \emph{The computation of some $p$-adic integrals in a motivic
  setting}, preprint, arXiv, 2008.

\bibitem[R{\"o}k08b]{Edinburgh}
\bysame, \emph{A motivic version of $p$-adic integration},
  http://www.icms.org.uk/downloads/motint/Rokaeus.pdf (2008), Talk given at
  ``Motivic Integration and its Interactions with Model Theory and
  Non-Archimedean Geometry'', ICMS, May 2008.

\bibitem[Seb04]{MR2075915}
Julien Sebag, \emph{Int\'egration motivique sur les sch\'emas formels}, Bull.
  Soc. Math. France \textbf{132} (2004), no.~1, 1--54. \MR{MR2075915
  (2005e:14017)}

\end{thebibliography}
\bibliographystyle{amsalpha}

\end{document}